\documentclass[11pt]{amsart}
\usepackage{amsfonts, amssymb, amsmath, mathtools, dsfont}\usepackage[linktoc=page, colorlinks, linkcolor=blue, citecolor=blue]{hyperref}
\usepackage{enumerate, commath}
\usepackage{graphicx}

\numberwithin{equation}{section}

\DeclareMathOperator{\E}{\mathbb{E}}

\DeclareMathOperator{\Span}{span}

\DeclareMathOperator{\sign}{sign}

\renewcommand{\Pr}[2][]{\mathbb{P}_{#1} \left\{ #2 \rule{0mm}{3mm}\right\}}

\newcommand{\ip}[2]{\langle#1,#2\rangle}
\newcommand{\Bigip}[2]{\Big\langle#1,#2\Big\rangle}

\def \P {\mathbb{P}}
\def \R {\mathbb{R}}

\def \e {\varepsilon}
\def \d {\delta}

\def \s {\sigma}

\def \tran {\mathsf{T}}

\def \one {{\textbf 1}}

\def \Id {\mathrm{Id}}

\def \H {\mathcal{H}}

\newtheorem{theorem}{Theorem}[section]

\newtheorem{lemma}[theorem]{Lemma}

\theoremstyle{definition}
\newtheorem{definition}[theorem]{Definition}
\newtheorem{question}[theorem]{Question}

\theoremstyle{remark}

\newtheorem{example}[theorem]{Example}

\begin{document}

\title{Are most Boolean functions determined \\ by low frequencies?}

\author[R. Vershynin]{Roman Vershynin}
\address{Department of Mathematics, University of California, Irvine}
\email{rvershyn@uci.edu}
\thanks{Partially supported by NSF DMS-1954233, NSF DMS-2027299, U.S. Army 76649-CS, and NSF+Simons Research Collaborations on the Mathematical and Scientific Foundations of Deep Learning.}


\begin{abstract}
	We ask whether most Boolean functions are determined by their low frequencies. We show a partial result: for almost every function $f: \{-1,1\}^p \to \{-1,1\}$ 
  there exists a function $f': \{-1,1\}^p \to (-1,1)$
  that has the same frequencies as $f$ up to dimension $(1/2-o(1))p$.
\end{abstract}

\maketitle

\section{Introduction}

\subsection{Uniqueness of Boolean data}

When is a high-dimensional distribution determined by its low-dimensional marginals?
To be specific, consider a random vector $X = (X_1,\ldots,X_p)$ that takes values in $\{0,1\}^p$.
We may wonder if there exist a random vector $Y = (Y_1,\ldots,Y_p)$ that also takes values in $\{0,1\}^p$, whose all marginal distributions up to a given dimension $d<p$ are the same as those\footnote{By this we mean that for any subset of coordinates $J$ or cardinality $\abs{J} \le d$, the distributions of the random vectors $(X_j)_{j \in J}$ and $(Y_j)_{j \in J}$ are the same.} of $X$, but whose distribution is different from that of $X$ everywhere on the cube.\footnote{By this we mean 
that $\Pr{X = \theta} \ne \Pr{Y = \theta}$ for all $\theta \in \{0,1\}^p$.} If this does happen, we call the distribution of $X$ {\em non-unique} with respect to marginals up to dimension $d$.

\smallskip

Some distributions are very much unique. For example, it is easy to check the following:

\begin{example}	\label{ex: unique}
  If $X_1=\cdots=X_p$ almost surely, then the distribution of $X$ is unique with respect to marginals up to dimension $2$.
\end{example}

However, we will show that ``most'' distributions on the cube are not unique with respect to marginals up to almost half the dimension:

\begin{theorem}	\label{thm: non-uniqueness}
  If $p$ is sufficiently large, then for most of the subsets $S \subset \{0,1\}^p$ the uniform distribution on $S$ is non-unique 
  with respect to the marginals of dimension $0.49 p$.
\end{theorem}

As we will see from the proof, ``most'' means all but at most $2^{c2^p}$ subsets where $c \in (0,1)$ is an absolute constant, and the number $0.49$ can be replaced by any constant smaller than $1/2$.

We will prove Theorem~\ref{thm: non-uniqueness} by showing that for most $S$, there exists a random vector $Y = Y(S) \in \{0,1\}^p$ 
that has the same marginals  up to dimension $0.49p$ as random vector $X$ uniformly distributed on $S$, and yet
$$
0 < \Pr{Y = \theta} < \frac{1}{\abs{S}} \quad \text{for all } \theta \in \{0,1\}^p.
$$
This yields non-uniqueness, since $\Pr{X = \theta}$ takes values $0$ and $1/\abs{S}$ only.
 
\subsection{Uniqueness of Boolean functions}
Theorem~\ref{thm: non-uniqueness} can be restated in a functional form. It is equivalent to saying that for most Boolean functions
$f: \{0,1\}^p \to \{0,1\}$ there exists a function $f': \{0,1\}^p \to (0,1)$ 
such that $f$ and $f'$ have the same marginals\footnote{This means that for any given set of coordinates 
$J \subset [p]$ with $\abs{J} \le 0.49p$ and any given values $(\tau_j)_{j \in J}$, we have 
$\sum_\theta f(\theta) = \sum_\theta f'(\theta)$ where the summation is over all $\theta \in \{0,1\}^p$
whose values on the coordinates in $J$ equal $(\tau_j)$.}
up to dimension $0.49p$. (To see the connection, let $f(\theta)$ be the indicator function of $S$ and set $f'(\theta) = \abs{S} \Pr{Y=\theta}$.)

Furthermore, by translation we can replace $0$ with $-1$ everywhere in the previous paragraph.
This allows to put Theorem~\ref{thm: non-uniqueness} in the context of Fourier analysis on the Boolean cube. 
Recall that {\em Rademacher functions} $r_j : \{-1,1\}^p \to \{-1,1\}$ are defined as 
$$
r_j(\theta) = \theta_j, \quad j=1,\ldots,p.
$$
{\em Walsh functions} $w_J : \{-1,1\}^p \to \{-1,1\}$ are
indexed by subsets $J \subset [p]$ and are defined as 
\begin{equation}	\label{eq: Walsh}
w_J = \prod_{j \in J} r_j,
\end{equation}
with the convention $w_\emptyset = 1$. The canonical inner product on the Boolean cube is 
defined as 
$$
\ip{f}{g}_{L^2} = \frac{1}{2^p} \sum_{\theta \in \{-1,1\}^p} f(\theta) \, g(\theta).
$$
Walsh functions form an orthonormal basis of $L^2(\{-1,1\}^p)$. 

In Section~\ref{s: proof of frequencies}, we shall prove:

\begin{theorem}	\label{thm: frequencies}
  If $p$ is sufficiently large, then 
  for most of the functions $f: \{-1,1\}^p \to \{-1,1\}$
  there exists a function $f': \{-1,1\}^p \to (-1,1)$
  that has the same frequencies\footnote{This means that for any set
  	$J \subset [p]$ with $\abs{J} \le 0.49p$, we have $\ip{f}{w_J} = \ip{f'}{w_J}$.}
  as $f$ up to dimension $0.49p$.
\end{theorem}

To see how Theorem~\ref{thm: frequencies} yields Theorem~\ref{thm: non-uniqueness} in its functional form,
note that the set of frequencies of $f$ up to dimension $d$ uniquely determines the set of marginals of $f$ up to dimension $d$. For simplicity, consider the two-dimensional marginal of $f$ corresponding to setting the first coordinate to $1$ and second to $-1$. We claim this marginal can be expressed in terms of the frequencies $\ip{f}{w_J}$ up to dimension $\abs{J} \le 2$. To do so, consider the set $\Theta = \{\theta \in \{-1,1\}^p:\; \theta_1=1, \, \theta_2=-1\}$; then the marginal is 
\begin{align*} 
\frac{1}{2^p} \sum_{\theta \in \Theta} f(\theta) 
  = \ip{f}{\one_\Theta}
  &= \Bigip{f}{\Big( \frac{1+r_1}{2} \Big) \Big( \frac{1-r_2}{2} \Big)} \\
  &= \frac{1}{4} \Big( \ip{f}{1} + \ip{f}{r_1} - \ip{f}{r_2} - \ip{f}{r_1 r_2}\Big) \\
  &= \frac{1}{4} \Big( \ip{f}{w_\emptyset} + \ip{f}{w_{\{1\}}} - \ip{f}{w_{\{2\}}} - \ip{f}{w_{\{1,2\}}}\Big).
\end{align*}
More generally, consider the marginal of dimension $d$ corresponding to setting a given set of coordinates $\theta_{i_1}, \ldots, \theta_{i_d}$ to given numbers $\tau_1, \ldots, \tau_d$. Consider the set $\Theta = \{\theta \in \{-1,1\}^p:\; \theta_{i_1}=\tau_1, \ldots, \theta_{i_d}=\tau_d\}$; then the marginal is 
$$
\frac{1}{2^p} \sum_{\theta \in \Theta} f(\theta) 
  = \ip{f}{\one_\Theta}
  = \Bigip{f}{\Big( \frac{1+\tau_1 r_{i_1}}{2} \Big) \cdots \Big( \frac{1+\tau_d r_{i_1}}{d} \Big)}.
$$
Expanding the right hand side, we can express it in terms of the frequencies $\ip{f}{w_J}$ up to dimension $\abs{J} \le d$. This shows that Theorem~\ref{thm: frequencies} yields Theorem~\ref{thm: non-uniqueness} indeed.

\begin{question}	\label{q: Boolean}
	Can we replace the range $(-1,1)$ of $f'$ by $\{-1,1\}$ in Theorem~\ref{thm: frequencies}? In other words, is it true that most Boolean functions are not determined by frequencies up to almost half the dimension? Is half dimension an optimal threshold? 
\end{question}

\subsection*{Acknowledgement}
The author is grateful for Paata~Ivanisvili for enlightening discussions, and to Xinyuan~Xie for spotting a minor inaccuracy in the original proof.

\section{Background}

Our proof of Theorem~\ref{thm: frequencies} is
based on one combinatorial result about hyperplane arrangements
and one probabilistic result -- a version of Rudelson's sampling theorem.

\subsection{Hyperplane arrangements}

A hyperplane arrangement is a collection of $N$ hyperplanes in $\R^n$. 
Removing these hyperplanes from $\R^n$ leaves an open set, 
and the connected components of this set are called {\em regions}. 
Counting the regions of a given hyperplane arrangement is a well studied 
problem in enumerative combinatorics, see e.g. \cite{buck1943partition, winder1966partitions, zaslavsky1975facing, stanley2004introduction}. 

For our purposes, two simple observations of Yu.~Zuev \cite{zuev1991combinatorial} 
will be sufficient. Fix an arrangement of $N$ hyperplanes in $\R^n$. 
The intersection of any subfamily of these original hyperplanes is called 
{\em an intersection subspace}. The dimension of an intersection subspace 
can range from zero (a single point) to $n$, since intersecting an empty set of hyperplanes yields the entire space $\R^n$.

\begin{lemma}[Zuev \cite{zuev1991combinatorial}]		\label{lem: Zuev}
 For any hyperplane arrangement, the number of regions is bounded below 
 by the number of intersection subspaces.
\end{lemma}

For example, an arrangement of $N$ hyperplanes in general position in $\R^n$ produces exactly $\binom{N}{\le n}$ intersection subspaces, 
since intersecting any subset of such hyperplanes of cardinality at most $n$ yields a different subspace. Zuev's Lemma~\ref{lem: Zuev} implies that such hyperplane arrangement has at least $\binom{N}{\le n}$ regions. This bound is in fact an identity, 
since any arrangement of $N$ hyperplanes in $\R^n$ has at most 
$\binom{N}{\le n}$ regions \cite{buck1943partition}; see also \cite[Proposition~2.4]{stanley2004introduction}.

In general, it could be hard to count intersection subspaces directly. 
This task, however, can be facilitated by the following simple observation. 
For convenience, let us focus here on {\em central} hyperplane arrangements,
those where all the hyperplanes pass through the origin. 
A central arrangement can be expressed in the form\footnote{Throughout the paper, $x^\perp$ denotes the hyperplane orthogonal to the vector $x$ in $\R^n$. More generally, for a subset $E \subset \R^n$,  we denote by $E^\perp$ the set of vectors in $\R^n$ that are orthogonal to all vectors in $E$. In particular, if $E$ is a linear subspace, $E^\perp$ is its orthogonal complement.}
$\{x_1^\perp,\ldots,x_N^\perp\}$ where 
$x_i \in \R^n$ is a vector orthogonal to the $i$-th hyperplane.

\begin{definition}[Resilience]
	Fix a system of vectors $x_1,\ldots,x_N \in \R^n$.
	We call a subset $I \subset [N]$ {\em resilient} 
	if the linear span of $\{x_i\}_{i \in I}$ does not contain any vector from $\{x_i\}_{i \in I^c}$.
\end{definition}

\begin{lemma}[Implicit in Zuev \cite{zuev1991combinatorial}]		\label{lem: intersections via resilience}
  For any central hyperplane arrangement $\{x_1^\perp,\ldots,x_N^\perp\}$, 
  the number of intersection subspaces is bounded below
  by the number of resilient subsets.
\end{lemma}

\begin{proof}
Any pair of distinct resilient subsets $I$ and $J$ satisfies 
$$
\Span\{x_i\}_{i \in I} \ne \Span\{x_j\}_{j \in J},
$$
since for any $i_0 \in I \setminus J$, the definition of resilience 
yields $x_{i_0} \not\in \Span\{x_j\}_{j \in J}$.
Hence the orthogonal complements of $\Span\{x_i\}_{i \in I}$ and $\Span\{x_j\}_{j \in J}$ are different, which we can write as
$$
\bigcap_{i \in I} x_i^\perp \ne \bigcap_{j \in j} x_j^\perp.
$$
Each side of this relation defines an intersection subspace. 
Thus, we obtained an injection 
from resilient subsets to intersection subspaces. 
The proof is complete.
\end{proof}

\subsection{Sampling}

In addition to hyperplane arrangements, our argument uses a version of Rudelson's sampling theorem \cite{rudelson1999random}; see also \cite[Theorem~5.6.1]{vershynin2018high}. The version we need can be most conveniently deduced from {\em matrix Bernstein inequality}, which is due to J.~Tropp \cite{tropp2012user}; see also \cite[Theorem~5.4.1]{vershynin2018high} for an exposition.

\begin{theorem}[Matrix Bernstein's inequality]		\label{thm: matrix Bernstein}
  Let $Z_1,\ldots,Z_N$ be independent, mean zero, $k \times k$ symmetric random matrices, 
  such that $\|Z_i\| \le M$ almost surely for all $i$.
  Then, for every $t \ge 0$, we have
  $$
  \P \Big\{ \norm[3]{\sum_{i=1}^N Z_i} \ge t \Big\} 
  \le 2k \exp \Big( -\frac{t^2/2}{\s^2 + Mt/3} \Big),
  $$
  where $\s^2 = \norm{\sum_{i=1}^N \E Z_i^2}$.
  Here $\norm{\cdot}$ denotes the operator norm of a matrix.
\end{theorem}

We are ready to state and prove a version of Rudelson's sampling theorem.

\begin{theorem}[Sampling with independent selectors]		\label{thm: sampling with selectors}
  Let $x_1,\ldots,x_K \in \R^k$ be vectors satisfying
  $$
  \frac{1}{K} \sum_{i=1}^K x_i x_i^\tran = \Id
  $$
  and such that $\norm{x_i}_2 \le 10\sqrt{k}$ for all $i$.
  Let $K \ge N \ge C k \log k$ where $C$ is a sufficiently large absolute constant.
  Consider independent Bernoulli random variables $\d_1,\ldots,\d_K$ satisfying 
  $\E \d_i = N/K$. Then 
  $$
  0.99 \cdot \Id \preceq \frac{1}{N} \sum_{i=1}^K \d_i x_i x_i^\tran \preceq 1.01 \cdot \Id
  $$
  with probability at least $1-\frac{1}{4} k^{-10}$.
\end{theorem}

\begin{proof}
We are going to apply matrix Bernstein inequality for the random matrices
$$
Z_i \coloneqq (\d_i-\d) x_i x_i^\tran, 
\quad \text{where } 
\d \coloneqq \E \d_i = \frac{N}{K}.
$$
By assumption, we have
$$
\norm{Z_i} 
= \norm[1]{x_i x_i^\tran}
= \norm{x_i}_2^2
\le 100k
\eqqcolon M.
$$
Furthermore, 
$$
0 \preceq \E Z_i^2 
= \d(1-\d) \norm{x_i}_2^2\, x_i x_i^\tran 
\preceq 100 \d k \cdot x_i x_i^\tran.
$$
Thus 
$$
\s^2 
= \norm[3]{\sum_{i=1}^N \E Z_i^2} 
\le 100 \d k \norm[3]{\sum_{i=1}^K x_i x_i^\tran}
= 100 \d k K 
= 100 k N.
$$
Applying matrix Bernstein inequality (Theorem~\ref{thm: matrix Bernstein}) 
and using the bounds on $M$ and $\sigma$, we obtain 
\begin{align*} 
\Pr{ \norm[3]{\frac{1}{N} \sum_{i=1}^N \d_i x_i x_i^\tran - \Id} \ge 0.01}
  &= \Pr{ \norm[3]{\sum_{i=1}^N Z_i} \ge 0.01N} \\
  &\le 2k \cdot \exp \Big( - \frac{cN}{k} \Big)
  \le \frac{1}{4} k^{-10}
\end{align*}
where the last inequality is guaranteed if $N \ge C k \log k$ with a sufficiently 
large absolute constant $C$.
This completes the proof.
\end{proof}

We will need a version of sampling theorem for a similar but not identical model of {\em sampling without replacement}.

\begin{theorem}[Sampling without replacement]		\label{thm: sampling without replacement}
  Let $x_1,\ldots,x_K \in \R^k$ be vectors satisfying
  $$
  \frac{1}{K} \sum_{i=1}^K x_i x_i^\tran = \Id
  $$
  and such that $\norm{x_i}_2 \le 10\sqrt{k}$ for all $i$.
  Let $N \ge C k \log k$ where $C$ is a sufficiently large absolute constant.
  Let $I$ be a random subset of $[K]$ with cardinality $\abs{I} = N$. 
  Then
  $$
  0.9 \cdot \Id \preceq \frac{1}{N} \sum_{i \in I} x_i x_i^\tran \preceq 1.1 \cdot \Id
  $$
  with probability at least $1-k^{-10}$.
\end{theorem}

\begin{proof}
Apply Theorem~\ref{thm: sampling with selectors} for $0.99N$ instead of $N$. 
It follows that a random set 
$$
I_0 \coloneqq \left\{ i:\; \d_i = 1 \right\} \subset [K]
$$
satisfies 
\begin{equation}	\label{eq: I0 succeq}
\frac{1}{0.99N} \sum_{i \in I_0} x_i x_i^\tran
\succeq 0.99 \cdot \Id
\end{equation}
with probability at least $1-\frac{1}{4} k^{-10}$.

Since $\E \d_i = 0.99N/K$, the expected cardinality of the set $I_0$ is $\E \abs{I_0} = 0.99N$.
Moreover, Chernoff inequality (see e.g. Exercise 2.3.5 in my book) implies that $\abs{I_0}$ 
is concentrated around its expectation, and in particular
\begin{equation}	\label{eq: I0 card}
\abs{I_0} \le N
\end{equation}
with probability at least $1-2e^{-cN}$.

Let us create a random set $I$ of cardinality exactly $N$ from $I_0$ by the following rule. 
If $\abs{I_0} < N$, add to $I_0$ exactly $N-\abs{I_0}$ elements chosen from $[K] \setminus I_0$ at random and without replacement. If $\abs{I_0} > N$, remove from $I_0$ exactly $\abs{I_0}-N$ elements chosen at random and without replacement. Clearly, $I$ obtained this way is a random subset of $[K]$ of cardinality $\abs{I} = N$.

Suppose $I_0$ satisfies both \eqref{eq: I0 succeq} and \eqref{eq: I0 card}; 
this occurs with probability at least $1-\frac{1}{4} k^{-10} - 2e^{-cN} \ge 1-\frac{1}{2} k^{-10}$.
In this case, $I \supset I_0$ and so
$$
\frac{1}{N} \sum_{i \in I} x_i x_i^\tran
\succeq \frac{1}{N} \sum_{i \in I_0} x_i x_i^\tran
\succeq 0.99^2 \cdot \Id
\succeq 0.9 \cdot \Id.
$$

A similar argument but for $1.01N$ instead of $0.99N$ yields 
$$
\frac{1}{N} \sum_{i \in I} x_i x_i^\tran
\preceq (1.01)^2 \cdot \Id 
\preceq 1.1 \cdot \Id
$$
with probability at least $1-\frac{1}{2} k^{-10}$. 
Taking the intersection of the two bounds completes the proof. 
\end{proof}

\section{Proof of Theorem~\ref{thm: frequencies}}	\label{s: proof of frequencies}

\subsection{Reduction to sign patterns}		\label{s: reduction to sig patterns}

Fix $d<p$. Let us call a function  
$f: \{-1,1\}^p \to \{-1,1\}$ {\em non-unique} if there exists a function $f': \{-1,1\}^p \to (-1,1)$
that has the same frequencies as $f'$ up to dimension $d$. 
Our goal is to show that most of the functions $f$ are non-unique for $d = 0.49p$.

Consider the following linear subspace of real-valued functions on the Boolean cube:
\begin{equation}	\label{eq: H}
\H \coloneqq \left\{ h: \{-1,1\}^p \to \R: \; \ip{h}{w_J} = 0 \quad \forall J \subset [p],\; \abs{J} \le d \right\}.
\end{equation}

\begin{lemma}	\label{lem: H}
  $f$ is non-unique if $f \equiv \sign h$ for some $h \in \H$.
\end{lemma}

\begin{proof}
Suppose $f \equiv \sign h$ for some $h \in \H$.
Let $\e>0$ be small enough and set $f' \coloneqq f-\e h$.
By definition of $\H$, the functions $f$ and $f'$ have the same frequencies 
up to dimension $d$. 
Moreover, if $f(\theta)=1$ then $h(\theta)>0$ and hence $f'(\theta)<1$.
Similarly, if $f(\theta)=-1$ then $h(\theta)<0$ and hence $f'(\theta)>-1$.
Therefore, if $\e$ is sufficiently small, $f'(\theta) \in (-1,1)$ for any $\theta$.
\end{proof}

Lemma~\ref{lem: H} shows that the number of non-unique functions $f$ 
is bounded below by the number of functions of the form $\sign h$ where $h \in \H$, 
which we call {\em sign patterns} generated by $\H$.
Hence, in order to show that most of the functions $f$ are non-unique, 
it suffices to show that the subspace $\H$ generates a lot of sign patterns. 

\subsection{From sign patterns to hyperplane arrangements}	\label{s: from signs to arrangements}
To count different sign patterns generated by $\H$, 
we express this as a problem about hyperplane arrangements.

Let $e_\theta : \{-1,1\}^p \to \{0,1\}$ denote the point evaluation function at $\theta \in \{-1,1\}^p$.
Consider two functions $h, h' \in \H$. We have $\sign h \not\equiv \sign h'$ if and only if 
there exists $\theta \in \{-1,1\}^p$ such that 
$\ip{h}{e_\theta}$ and $\ip{h'}{e_\theta}$ have opposite signs. This happens if and only if 
$h$ and $h'$ are separated by at least one hyperplane $e_\theta^\perp \cap \H$ in $\H$.
The latter is equivalent to $h$ and $h'$ lying in different regions of the hyperplane arrangement 
$\{e_\theta^\perp \cap \H\}_{\theta \in \{-1,1\}^p}$ in the subspace $\H$.
Therefore, the number of different sign patterns generated by $\H$ 
(and thus also the number non-unique functions $f$) is bounded
below by the number of regions of the hyperplane arrangement
$\{e_\theta^\perp \cap \H\}_{\theta \in \{-1,1\}^p}$ in the subspace $\H$.

Denoting by $P_\H$ the orthogonal projection onto $\H$, we see that the vector
$P_\H e_\theta$ is orthogonal to the hyperplane $e_\theta^\perp \cap \H$ in $\H$.
Now apply Lemmas~\ref{lem: Zuev} and \ref{lem: intersections via resilience} 
for our hyperplane arrangement in $\H$. We conclude that the number of regions 
(and thus also the number of sign patterns generated by $\H$, and thus also the number
of non-unique functions $f$) is bounded below by the 
number of resilient subsets for the collection 
$\{P_\H e_\theta\}_{\theta \in \{-1,1\}^p}$.

\subsection{A necessary condition for non-resilience}		\label{s: resilience necessary}

To complete the proof, we will show that most of the subsets are resilient.
Let us see what happens when the complement $\Theta^c$ of some subset
$\Theta \subset \{-1,1\}^p$ is {\em not} resilient.
By definition, there exists $\theta_0 \in \Theta$ such that
$$
P_\H e_{\theta_0} \in \Span \{P_\H e_\theta\}_{\theta \in \Theta^c}.
$$
This means that there exist real numbers $(a_\theta)_{\theta \in \Theta^c}$ such that
$$
P_\H e_{\theta_0} = \sum_{\theta \in \Theta^c} a_\theta P_\H e_\theta.
$$
This in turn means that 
$$
e_{\theta_0} - \sum_{\theta \in \Theta^c} a_\theta e_\theta
\in \H^\perp 
= \Span \left\{ w_J: J \subset [p],\; \abs{J} \le d \right\}.
$$
Therefore, there exist real numbers $(b_J)_{J \subset [p],\; \abs{J} \le d}$ such that
$$
e_{\theta_0} - \sum_{\theta \in \Theta^c} a_\theta e_\theta
= \sum_{J \subset [p],\; \abs{J} \le d} b_J w_J.
$$
If we evaluate this identity at any point $\theta \in \Theta \setminus \{\theta_0\}$, 
the left hand side of it vanishes, and we have
$$
\sum_{J \subset [p],\; \abs{J} \le d} b_J w_J(\theta) = 0.
$$
Express this identity this as an orthogonality relation in $\R^{\binom{p}{\le d}}$, namely 
$$
\ip{b}{w(\theta)} = 0
$$
where 
\begin{equation}	\label{eq: Rpd}
b = (b_J)_{J \subset [p],\; \abs{J} \le d}
\quad \text{and} \quad
w(\theta) \coloneqq (w_J(\theta))_{J \subset [p],\; \abs{J} \le d}.
\end{equation}

Summarizing, we showed the following. 

\begin{lemma}[A necessary condition for non-resilience]	\label{lem: resilience necessary}
  If $\Theta^c$ is {\em not} resilient for some $\Theta \subset \{-1,1\}^p$
  then there exists $\theta_0 \in \Theta$ and $b \in \R^{\binom{p}{\le d}}$ such that 
  $$
  \ip{w(\theta)}{b} = 0
  \quad \text{for all } \theta \in \Theta \setminus \{\theta_0\}.
  $$
\end{lemma}

\subsection{Applying the sampling theorem}		\label{s: applying Rudelson}

\begin{lemma}		\label{lem: lower}
  Assume that $2^p \ge C \binom{p}{\le d} \log \binom{p}{\le d}$ 
  where $C$ is a sufficiently large absolute constant.
  Let $\Theta$ be a random subset of $\{-1,1\}^p$
  with cardinality $\abs{\Theta} = 0.1 \cdot 2^p$. 
  Then, with probability at least $1-p^{-10}$, the following uniform lower bound holds:
  \begin{equation}			\label{eq: lower}
  \frac{1}{0.1 \cdot 2^p} \sum_{\theta \in \Theta} \ip{w(\theta)}{b}^2 
  \ge 0.9 \norm{b}_2^2
  \quad \text{for all } b \in \R^{\binom{p}{\le d}}.
  \end{equation}
\end{lemma}

\begin{proof}
Orthonormality of Walsh basis \eqref{eq: Walsh} can be expressed as 
$$
\frac{1}{2^p} \sum_{\theta \in \{-1,1\}^p} w_J(\theta) \, w_{J'}(\theta)
= \delta_{J, J'}
\quad \text{for any } J, J' \subset [p].
$$
Using our notation \eqref{eq: Rpd}, this becomes
$$
\frac{1}{2^p} \sum_{\theta \in \{-1,1\}^p} w(\theta) \, w(\theta)^\tran 
= \Id.
$$
(The identity on the right side is in $\R^{\binom{p}{\le d}}$.)
Moreover, since all coordinates of $w(\theta)$ are $\pm 1$, we have
\begin{equation}	\label{eq: w norm}
\norm{w(\theta)}_2^2 = \binom{p}{\le d} 
\quad \text{for all } \theta.
\end{equation}

Apply Rudelson's Sampling Theorem \ref{thm: sampling without replacement}
for $k = \binom{p}{\le d}$ and $N = 0.1 \cdot 2^p$.
We get
$$
\frac{1}{0.1 \cdot 2^p} \sum_{\theta \in \Theta} w(\theta) w(\theta)^\tran 
\succeq 0.9 \cdot \Id
$$
with probability at least $1-k^{-10} \ge 1 - p^{-10}$. 
Multiplying by $b^\tran$ on the left and by $b$ on the right, we obtain \eqref{eq: lower}.
\end{proof}

\subsection{Completion of the proof}		\label{s: completion}

If $1 \le d \le 0.49p$ then $2^p \ge C \binom{p}{\le d} \log \binom{p}{\le d}$, so the assumptions of Lemma~\ref{lem: lower} hold. Thus, the uniform lower bound \eqref{eq: lower} holds for most of the subsets $\Theta \subset \{-1,1\}^p$ with cardinality $0.1 \cdot 2^p$. Moreover, since adding more points to the set $\Theta$ can only make the left hand side of \eqref{eq: lower} larger,
we conclude that \eqref{eq: lower} holds for most sets $\Theta$ of cardinality $0.1 \cdot 2^p$ and all their supersets, which is most of the subsets of the cube. 

We claim that for any subset $\Theta$ satisfying \eqref{eq: lower}, 
the complement $\Theta^c$ is resilient. This will be enough to complete the proof of Theorem~\ref{thm: frequencies}.
Indeed, it would follow that most of the subsets of $\{-1,1\}^p$
of cardinality $0.9 \cdot 2^p$ are resilient. 
Since a subset of a resilient subset is resilient,  
most of the subsets of $\{-1,1\}^p$ of cardinality bounded by $0.9 \cdot 2^p$ are resilient.
Thus, most of the subsets of $\{-1,1\}^p$ are resilient, completing the proof.

Assume for contradiction that $\Theta^c$ is {\em not} resilient. 
Then the necessary condition (Lemma~\ref{lem: resilience necessary})
implies the existence of 
$\theta_0 \in \Theta$ and $b \in \R^{\binom{p}{\le d}}$ such that 
\begin{equation}	\label{eq: sum w theta zero}
\sum_{\theta \in \Theta \setminus \{\theta_0\}} \ip{w(\theta)}{b}^2 = 0.
\end{equation}
On the other hand, 
$$
\sum_{\theta \in \Theta \setminus \{\theta_0\}} \ip{w(\theta)}{b}^2
= \sum_{\theta \in \Theta} \ip{w(\theta)}{b}^2 - \ip{w(\theta_0)}{b}^2.
$$
By Lemma~\ref{lem: lower}, the first term on the right hand side is 
bounded below by $0.1 \cdot 2^p \cdot 0.9 \norm{b}_2^2$.
The second term on the right hand side is bounded above by
$$
\norm{w(\theta_0)}_2^2 \cdot \norm{b}_2^2 
= \binom{p}{\le d} \norm{b}_2^2
$$
due to \eqref{eq: w norm}. Therefore we have 
\begin{equation}	\label{eq: sum w theta positive}
\sum_{\theta \in \Theta \setminus \{\theta_0\}} \ip{w(\theta)}{b}^2 > 0
\end{equation}
as long as 
$$
0.1 \cdot 2^p \cdot 0.9> \binom{p}{\le d},
$$
which is true if $d = 0.49p$ if $p$ is sufficiently large.
The contradiction of \eqref{eq: sum w theta zero} and \eqref{eq: sum w theta positive}
completes the proof of Theorem~\ref{thm: frequencies}. \qed

\bibliographystyle{plain}
\bibliography{nonuniqueness.bib}

\end{document}